\documentclass[12pt]{amsart}
\def \R{\mathbb{R}}
\def \C{\mathbb{C}}
\def \Z{\mathbb{Z}}
\def \co{\colon}
\def \id{\mathop\mathrm{id}}

\usepackage{graphicx} 
\usepackage{graphicx, amscd, soul}
\usepackage{xcolor}
\usepackage{psfrag}

\usepackage{geometry}
\usepackage{anysize} \marginsize{1in}{1in}{1in}{1in}

\usepackage{soul,xcolor}
\newcommand\erase{\bgroup\markoverwith{\textcolor{red}{\rule[.5ex]{2pt}{0.4pt}}}\ULon}

\newtheorem{theorem}{Theorem}[section]
\newtheorem{lemma}[theorem]{Lemma}
\newtheorem{corollary}[theorem]{Corollary}

\theoremstyle{definition}
\newtheorem{definition}[theorem]{Definition}

\newtheorem{remark}[theorem]{Remark}
\newtheorem{example}[theorem]{Example}

\newenvironment{dedication}
  {
   \itshape             
   \raggedright          
  }
  {
  }

\title{On the number of components of folds of image simple fold maps}
\author{Rustam Sadykov}
\author{Osamu Saeki}



\date{}

\begin{document}
\begin{dedication}
    This paper is dedicated to the memory \\ of our esteemed colleague Professor 
    \\ Maria Aparecida Soares Ruas. \vspace{3mm}
\end{dedication}
\begin{abstract}
    A smooth map between manifolds is said to be \emph{image simple} if its restriction to its 
    singular point set is a topological embedding. 
    We study the parity of the number of connected components of the singular point set for 
    image simple fold maps from a closed manifold of dimension $\ge 2$ to a surface. 
    It is known that this parity is a homotopy invariant when the source manifold is of even dimension
    and the target surface is orientable.
    In this paper we show that for an arbitrary image simple fold map from a closed odd-dimensional manifold 
    of dimension $\ge 3$ to a (possibly non-orientable) surface, this parity is not a homotopy invariant. 
    We give two constructive proofs: one uses
    techniques from open book decompositions and round fold maps, and the other uses allowable moves.
    The constructed homotopies add one new component to the singular point set.
    We also give a new proof of the homotopy invariance of the
    parity of the number of connected components of the singular point set for 
    image simple fold maps from a closed even-dimensional manifold to an orientable surface
    together with an example showing that this does not hold for maps into non-orientable
    surfaces in general.
\end{abstract}

\maketitle

\section{Introduction}\label{s:intro}

Given a smooth map $f\co M\to N$ between manifolds of dimensions $m$ and $n$, respectively, with $m\ge n$, 
we say that $x \in M$ is a \emph{singular point} if 
the differential 
$df_x \co T_xM \to T_{f(x)}N$ does not have the maximal rank $n$.
It is known that when $f$ is a \emph{generic} smooth map into a surface $N = S$, 
the set $\Sigma(f)$ of its singular points
consists of folds and cusps (defined in Section~\ref{s:preliminaries}), see \cite{T, Wh}. 
More specifically, for a given \emph{closed} manifold $M$ of dimension $\ge 2$
and a surface $S$,
the set of smooth maps $f \co M \to S$ which have only folds and cusps as their
singular points is open and dense in the space $C^\infty(M, S)$ of smooth maps of $M$ into $S$
endowed with the Whitney $C^\infty$ topology (for example, see \cite{GG}).
In this sense, generic maps to surfaces can be considered as generalizations of
Morse functions on manifolds which are generic maps into the $1$-dimensional
manifold $\R$. Therefore, it is expected that
the topology of the source manifold $M$ is somehow related to
various properties of generic maps $M \to S$ on that manifold.


In this paper, we focus on generic maps to surfaces without cusp singularities.
A smooth map $f \co M \to S$ that has only folds
as its singularities is called a \emph{fold map}.
Note that if $M$ is a closed manifold, 
then the singular point set $\Sigma(f)$ is a finite disjoint union of circles in $M$, 
and the restriction $f|_{\Sigma(f)}$ is an immersion. 
Among fold maps, we are mainly interested in 
those whose singular point sets have well-behaved images.
We say that a smooth map $f$ between manifolds is \emph{image simple} if the restriction of $f$ to the set 
$\Sigma(f)$ of its singular points is a topological embedding. 
In other words, the image $f(\Sigma(f))$ of the singular point set has no self-intersections 
in the target manifold. 
Image simple smooth generic maps
have recently attracted some detailed studies (see, for example, 
\cite{BSPNAS, BS1, Ki14a, Ki14b, Ki14c, Sae19, Sae25}). 

Thus, our main objective 
is to study certain topological properties of image simple fold maps from a closed manifold $M$ of dimension 
$m \ge 2$ to a surface $S$. The image $f(\Sigma(f))$ of the singular point set $\Sigma(f)$ of such a map
is a finite disjoint union of embedded closed curves in $S$. 
The main aim of the paper is to study one of the basic characteristics of an image simple fold map $f$, 
namely the parity of the number $\#|\Sigma(f)|$ of components of the singular point
set, and specifically, whether this parity is a homotopy invariant of $f$ or not.

\begin{example}\label{example1}
The restriction of the standard projection $\R^3\to \R^2$ to the standard unit 
sphere defines an image simple fold map $f \co S^2\to \R^2$ with $\#|\Sigma(f)|=1$. 
For any prescribed positive odd integer $k$,
it is easy to construct an image simple fold map $f_k$ of $S^2$ into $\R^2$ with $\#|\Sigma(f_k)| = k$; 
however, no such map exists for any positive even integer $k$
by \cite[Theorem~1.1]{KS} or by Theorem~\ref{thm:even} of the present paper.
In fact, according to \cite{Y}, there exists no fold map $g \co S^2 \to \R^2$
with $\# |\Sigma(g)|$ even, even if we allow $g$ not to be image simple.
See also Remark~\ref{rem:surface}.
\end{example}

The above question is a generalization of the first of 
the two Takase problems formulated in \cite[Problem 3.2]{Sae19}. 
Takase considered image simple generic maps $f\co S^3\to S^2$ without definite fold singular points
with $\Sigma(f) \neq \emptyset$.
(For the definition of a definite fold singularity, refer to Section~\ref{s:preliminaries}.)
The first Takase problem asks if the number $\#|\Sigma(f)|$ of components of $\Sigma(f)$ is 
congruent modulo two to $d+1$, 
where $d\in \Z\cong  \pi_3(S^2)$ is the homotopy class represented by $f$. 
The second problem asks if the minimum number of components of $\Sigma(f)$,
over all $f$ as above in the given homotopy class $d\in \Z\cong  \pi_3(S^2)$,
is equal to $|d|+1$. 


The second author \cite{Sae19a} later answered the first Takase problem in the negative by 
constructing two homotopic image simple fold maps $f$ and $g$ without definite folds such that 
$\#|\Sigma(f)|\not\equiv \#|\Sigma(g)| \pmod 2$, 
see Section~\ref{s:2a} for details.
Thus, in the case of image simple fold maps of manifolds of dimension $3$ to surfaces, 
the parity of the number of fold components is not a homotopy invariant. 
It is worth mentioning that the second author constructed an explicit generic
homotopy $F\co M\times [0,1]\to S\times [0,1]$ 
between $f$ and $g$. 
The singular point set of a generic
homotopy between fold maps to a surface is, in general, a smooth surface properly embedded in $M\times [0,1]$, 
and we see that for the specific generic homotopy $F$ constructed in \cite{Sae19a},
the surface $\Sigma(F)$ of the singular points of $F$ is the disjoint union of $\Sigma(f)\times [0,1]$ and a M\"obius band. 
Furthermore, the restriction of $F$ to the M\"obius band has a triple self-intersection point. 

Note that every homotopy between image simple fold maps can
be approximated arbitrarily well by a \emph{generic homotopy} (for details, see \S\ref{s:moves}).
Therefore, when two image simple fold maps are homotopic, we may assume 
that the homotopy is generic.

The mentioned results raised the broader question of identifying the precise conditions under which 
the parity of the number of fold components 
is a homotopy invariant. 
It was partially answered by the first author and Kahmeyer in \cite{KS}, who proved that 
for an image simple fold map $f\co M\to S$ of 
a closed manifold $M$ of dimension $\ge 2$ to a surface $S$, the parity of the number of 
connected components of $\Sigma(f)$ does not change under homotopy $F\co M\times [0,1]\to S\times [0,1]$ 
if any of the following conditions holds: 
\begin{enumerate}
\item the dimension of $M$ is even and $S$ is orientable (refer to
Theorem~\ref{thm:even} of the present paper as well),
\item $F$ is generic and
the singular point set $\Sigma(F)$ of the homotopy is an orientable surface, or 
\item $F$ is generic and the restriction $F|_{\Sigma(F)}$ of 
homotopy to its singular point set does not have triple self-intersection points, 
and either $M$ is a closed orientable manifold of dimension $3$ and $S$ is orientable, or
$M$ is a closed manifold of odd dimension $m>2$ and $S$ is $\R^2$ or $S^2$. 
\end{enumerate}

In the present paper we will study the cases excluded in the above result (1):
when the dimension of $M$ is odd and is at least $3$ and $S$ is an arbitrary surface, and when
the dimension of $M$ is even and $S$ is a non-orientable surface. 
Our first main result shows that in the odd-dimensional case, the parity is not a homotopy invariant
as follows.

\begin{theorem}\label{th:main}
Let $M$ be a closed odd-dimensional manifold with
$m = \dim M \ge 3$, and $S$ be a possibly non-orientable surface.
Then, for every image simple fold map $f\co M\to S$, there exists a homotopy $F$ connecting $f$ 
to another image simple fold map $g$ such that $\#|\Sigma(f)|$ and $\#|\Sigma(g)|$ have different parities. 
Furthermore, the homotopy $F$ can be chosen so that it is supported in an arbitrarily small 
$m$-disc in $M$.
\end{theorem}

In fact, we will construct an explicit 
generic homotopy $F\co M\times [0,1]\to S\times [0,1]$ 
satisfying the conclusion of Theorem~\ref{th:main}. 
Its singular point set $\Sigma(F)$ 
contains a M\"{o}bius band component, and $F|_{\Sigma(F)}$ 
has a triple self-intersection point.
This shows that conditions (2) and (3) above are necessary in general. 
On other hand, since Theorem~\ref{th:main} permits the target surface $S$ to be non-orientable, 
it is not guaranteed that every generic homotopy satisfying 
the conclusion of the theorem
must have a triple self-intersection point.
However, if we remove the condition that
the homotopy should be supported in a small disc, then, as we
will see below, in the case where $M$ is even-dimensional, and both $M$
and $S$ are non-orientable, the absence of a triple self-intersection
is not a necessary condition.

Our second main theorem is the following.

\begin{theorem}\label{thm:main2}
For every $n \geq 1$, there exist a closed $2n$-dimensional non-orientable
manifold $M$ and image simple fold maps $f$ and $g \co M \to \mathcal{M}$
to the open M\"obius band $\mathcal{M}$ such that $f$ and $g$ are homotopic, but
$\# |\Sigma(f)| \not\equiv \# |\Sigma(g)| \pmod{2}$.
\end{theorem}

We also construct an explicit generic
homotopy $F$ between $f$ and $g$ in Theorem~\ref{thm:main2}
such that the singular point set $\Sigma(F)$ is non-orientable, but that
it has no triple self-intersections.
This again shows that condition (2) above is necessary in general: however,
condition (3) is not necessary in general.

To summarize, let $M$ be a closed manifold of dimension at least $2$, and $S$ be a surface. 
Let $F$ be a (generic) homotopy between
two image simple fold maps $f$ and $g\co M\to S$. 
Suppose that the dimension of $M$ is even. 
\begin{itemize}
\item If $S$ is orientable, then the parity is a homotopy invariant by \cite{KS} 
or by Theorem~\ref{thm:even}.
\item If $S$ is non-orientable, then the parity is not a homotopy invariant by Theorem~\ref{thm:main2}. 
\item If $\# |\Sigma(f)| \not\equiv \# |\Sigma(g)| \pmod{2}$, then $\Sigma(F)$ is non-orientable, 
but $F|_{\Sigma(F)}$ may or may not have triple self-intersection points by \cite{KS} and 
by Theorem~\ref{thm:main2}. 
\end{itemize}

Suppose that the dimension $m$ of $M$ is odd. 
\begin{itemize}
\item The parity of the number of fold components 
is not a homotopy invariant in general by \cite{Sae19a} in the case $m=3$ and 
by Theorem~\ref{th:main} in the case $m\ge 3$. 
\item If $\Sigma(F)$ is an orientable surface, then the parity is an
invariant by \cite{KS}.
\end{itemize}

Furthermore, suppose that either $M$ is a closed orientable manifold of dimension $3$ and $S$ is orientable, or
$M$ is a closed manifold of odd dimension $m>2$ and $S$ is $\R^2$ or $S^2$. 
\begin{itemize}
\item If $\# |\Sigma(f)| \not\equiv \# |\Sigma(g)| \pmod{2}$, then $F|_{\Sigma(F)}$ 
has a triple self-intersection point by \cite{KS}. 
\end{itemize}
\setstcolor{blue}
Note that the last statement does not hold in general without the
assumption on $S$ according to Theorem~\ref{thm:main2}.

The paper is organized as follows. In Section~\ref{s:preliminaries} we prepare
some basic notions necessary in the paper.
In Section~\ref{s:moves}, for the reader's convenience, based on \cite{Cerf, HW, KS, Sae25},
we summarize the moves for generic maps together with sufficient
conditions for some of them to be allowable, which will be used in
later sections for constructing explicit homotopies.
In Section~\ref{s:obd} we use open book decompositions as well as round fold maps 
of odd-dimensional spheres to prove Theorem~\ref{th:main} 
producing a homotopic modification $g$ of $f$ with $\#|\Sigma(g)|\equiv \#|\Sigma(f)|+1 \pmod{2}$.  
In Section~\ref{s:2a} we give another proof of Theorem~\ref{th:main} along the argument in \cite{Sae19a},
using allowable moves (or always-realizable moves) as listed in Section~\ref{s:moves}.
In Section~\ref{s:even}, we first give a new proof of the homotopy invariance of the
parity of the number of connected components of the singular point set for 
image simple fold maps from a closed even-dimensional manifold to an orientable surface
(see Theorem~\ref{thm:even}).
Then, we prove Theorem~\ref{thm:main2} by constructing an explicit example.

Throughout the paper, manifolds and maps between them are smooth of
class $C^\infty$ unless otherwise specified. 
The symbol ``$\cong$'' denotes
an appropriate isomorphism between algebraic objects or a diffeomorphism between manifolds.

\section{Preliminaries}\label{s:preliminaries}

Let $M$ be a smooth closed manifold of dimension $m \geq 2$,
and $S$ a smooth surface.

\begin{definition}\label{def:fold}
A singular point $p$ of a smooth map $f \co M \to S$
that can be described by the normal form
$$(t, x_1, x_2, \ldots, x_{m-1})
\mapsto (X, Y) =
(t, -x_1^2 - \cdots -x_i^2 + x_{i+1}^2 + \cdots + x_{m-1}^2)$$
with respect to appropriate coordinates around $p$ and $f(p)$
is called a
\emph{fold singularity} (or a \emph{fold}), where the integer
$i$ is called the \emph{index} with respect to the $(-Y)$-direction
(see Fig.~\ref{fig82}).
When we choose the $Y$-coordinate so that
$0 \leq i \leq \left\lfloor (m-1)/2 \right\rfloor$ holds,
$i$ is called the \emph{absolute index}.
Note that the absolute index of a fold is well-defined and
does not depend on the choice of local coordinates.

When $i=0$ or $m-1$, the singular point is called a \emph{definite fold}.
Otherwise, it is called an \emph{indefinite fold}.

A smooth map $f \co M \to S$ that has only folds
as its singularities
is called a \emph{fold map}.
\end{definition}

\begin{figure}[t]
\centering
\includegraphics[width=0.9\textwidth,
keepaspectratio]{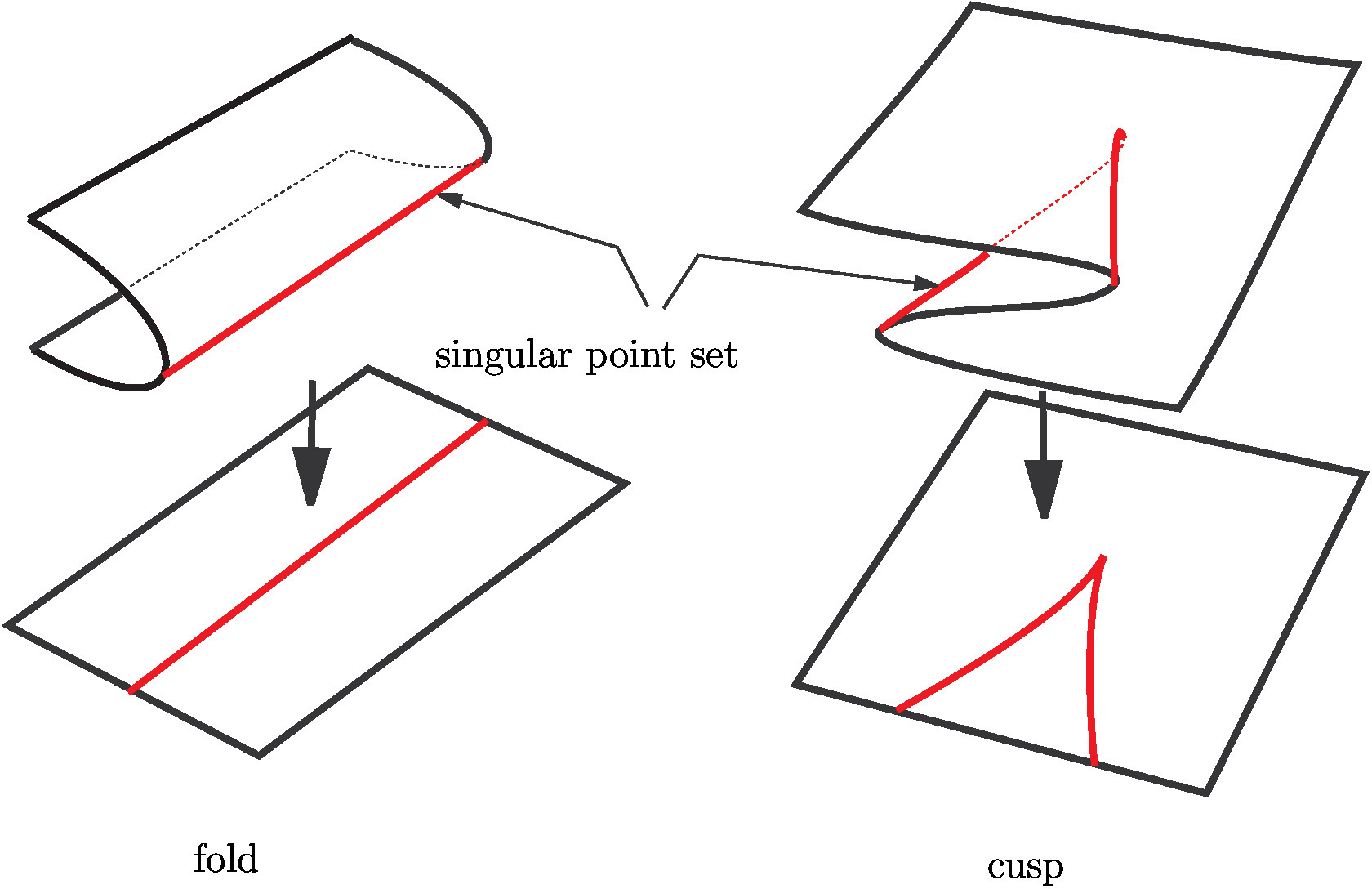}
\caption{Fold and cusp singularities for $m=2$}
\label{fig82}
\end{figure}

For a smooth map $f \co M \to S$, 
$\Sigma(f)$ denotes the set of singular points of $f$.
If $f$ is a fold map and the source manifold $M$ is closed
(compact and without boundary), then $\Sigma(f)$
is a disjoint union of finitely many circles, and
$f|_{\Sigma(f)}$ is an immersion.

\begin{definition}
A singular point of a smooth map $M \to S$
that has the normal form 
$$(t_1, t_2, x_1, x_2, \ldots, x_{m-2})$$
$$\mapsto (X, Y) =
(t_1, t_2^3 + t_1 t_2 -x_1^2 - \cdots -x_i^2 + x_{i+1}^2 + \cdots + x_{m-2}^2)$$
for some $0 \leq i \leq (m-2)/2$ is called a
\emph{cusp singularity} (or a \emph{cusp}).
See Fig.~\ref{fig82}.
\end{definition}

\begin{definition}
A fold map $f \co M \to S$ of a closed
manifold into a surface is said to be
\emph{image simple}
if $f|_{\Sigma(f)} \co \Sigma(f) \to S$ is a topological embedding. 
Since, for a fold map, the restriction to its singular point
set is an immersion, this is equivalent to requiring that
$f|_{\Sigma(f)}$ is a smooth embedding.
\end{definition}

Let $C_r=\{x\in \R^2 \, : \, ||x||=r \}$ denote the circle of radius $r>0$ centered at the origin. 
We say that a set $C$ of embedded circles in $\R^2$ is a finite collection of $s$ \emph{concentric circles} 
if there is a diffeomorphism of $\R^2$ that takes the collection 
$C$ to the disjoint union of circles $C_1, C_2, \ldots, C_s$. 
A \emph{round fold map} \cite{Ki14a, Ki14b, Ki14c} of a closed manifold $M$ is a 
fold map $f\co M\to \R^2$ whose restriction $f|_{\Sigma(f)}$ to the set of singular points is 
an embedding onto a finite collection of concentric circles in $\R^2$. 
Note that a round fold map is always image simple.

Let $F$ be a compact manifold of dimension $k$, and $\mu \co F \to F$ a diffeomorphism which
is the identity on the boundary $\partial F$. 
Then, we have a fiber bundle 
$$S^1 \tilde\times F = [0, 2\pi] \times F/\!\!\sim \,\, \to S^1,$$
where we define $(2\pi, x) \sim (0, \mu(x))$ 
for each $x \in F$, and where the projection map
is induced by the projection to the first factor $[0, 2\pi] \times F \to [0, 2\pi]$.
Note that the boundary of this bundle is a trivial fiber bundle $S^1 \times \partial F \to S^1$. 
Then, the space $M = S^1 \tilde\times F/\!\sim$
obtained from $S^1 \tilde\times F$ by identifying $S^1 \times \{y\}$ to a point for each $y \in \partial F$
has a natural structure of a smooth closed $(k+1)$-dimensional manifold.
Such a description of $M$ is called an \emph{open book decomposition} with \emph{monodromy} $\mu$.
Furthermore, the submanifold $\{t\} \tilde\times F$ for each $t \in S^1$ is called a \emph{page} and
the submanifold (diffeomorphic to $\partial F$) obtained from $S^1 \times \partial F$
by the identification is called the \emph{binding}.

Suppose that there exists a smooth family $f_t \co F \to [0, \infty)$
of Morse functions, $t \in S^1$, such that 
\begin{itemize}
\item $f_t(\partial F) = 0$, 
\item the critical values of $f_t$ are in $(0, \infty)$, 
\item the critical values are all distinct, and
\item $\tilde{f}|_{\Sigma(\tilde{f})}$ is locally constant, where $\tilde{f} \co S^1 \times F \to [0, \infty)$ 
is defined by $\tilde{f}(t, x)=f_t(x)$.
\end{itemize}
Furthermore, we assume that 
$f_{2 \pi} = f_0 \circ \mu$ and hence that
the family $\{f_t\}$ induces a smooth map $\beta \co S^1 \tilde{\times} F \to S^1 \times [0, \infty)$ 
defined by $\beta(t, x) = (t, f_t(x))$.
Then, the composition of $\beta$ with the natural map $S^1 \times [0, \infty) \to \R^2$
corresponding to the polar coordinates of $\R^2$ defines a 
smooth map $M \to \R^2$.
We see easily that the map thus constructed is a round fold map. 
Furthermore, the indices of its fold points with respect to the inward direction 
coincide with 
those of the corresponding critical points of the Morse functions $f_t$.

\section{Generic homotopies of maps to a surface}\label{s:moves}

Let $C^{\infty}(M, N)$ denote the space of smooth maps of a manifold $M$ into a manifold $N$. 
It is endowed with the Whitney $C^\infty$ topology. 
We say that a map $f\co M\to N$ is \emph{right-left equivalent} to a map 
$g \co M\to N$ if 
$g = \Phi \circ f \circ \Psi^{-1}$ for some diffeomorphisms $\Psi$ of $M$ and $\Phi$ of $N$.
The same definition applies to map germs. 
A map $f\co M\to N$ is said to be \emph{stable} if every map 
$g$ sufficiently close to $f$ 
in $C^\infty$ topology is right-left equivalent to $f$.
A map germ of $f\co M\to N$ is \emph{stable} at a point $p\in M$ if for every neighborhood $U$ 
of $p$ and every map $g \co M \to N$ sufficiently close to $f$ there is a point 
$p'$ in $U$ such that the map germ of $g$ at $p'$ is right-left equivalent to the map germ of $f$ at $p$
(for example, see \cite{AGV}).

Let $f\co M\to N$ be a smooth map. For a point $p$ in $M$, we define $\Sigma_p(f)$ as the set of points in 
$M$ at which the map germ of $f$ is right-left equivalent to the map germ of $f$ at $p$. 
Mather \cite{MaV} proved that if all map germs of $f$ are stable, then the sets 
$\Sigma_p(f)$ are submanifolds of $M$. 
For a finite family of subspaces $\{P_i\}_{i=1}^s$ of a vector space, we say that the spaces are 
in \emph{general position} if 
for every sequence of integers $1 \leq i_1 < i_2 < \cdots < i_r \leq s$, we have
$$\mathrm{codim}\,(P_{i_1} \cap P_{i_2} \cap \cdots \cap P_{i_r}) =
\sum_{k=1}^r \mathrm{codim}\, P_{i_k},$$
where ``$\mathrm{codim}$'' denotes the codimension.
Suppose that all map germs of $f$ are stable and for any finitely many points $p_1, p_2, \ldots, p_s$ 
with $s>1$ and $f(p_1)= f(p_2) = \cdots = f(p_s)=q$, the subspaces 
$d_{p_i}f(T_{p_i}\Sigma_{p_i}(f))$ are in general position
in $T_qN$.
Then we say that $f$ satisfies the \emph{Mather normal crossing condition}. 

In a series of papers, Mather investigated stability of maps and proved its equivalence to various 
other conditions. In particular, Mather proved~\cite{MaV} that a proper map $f$ is stable if and only if 
it satisfies the Mather normal crossing condition (for this particular interpretation of the Mather result, 
and a new proof, see {\cite{Sad25}}).
In the case of maps with fold and cusp singularities 
the Mather normal crossing condition 
simply requires that the images of submanifolds of fold and cusp singular points intersect 
(and self-intersect) generically. In particular, image simple fold maps are always stable.


To study homotopies of stable maps,
we work with generic homotopies.
 For the definition and some properties of
\emph{generic homotopies}, the reader is referred to \cite{KS}.
In particular, the class of generic homotopies is open and dense in the space of all homotopies 
(e.g., see \cite[Lemma 4.1]{KS}), which means that an arbitrary
homotopy between stable maps can be approximated arbitrarily well
by a generic homotopy. 
For this reason, we may and will almost always assume
that a homotopy between stable maps is generic.
Furthermore, under a generic homotopy, the image of the singular point
set is modified by isotopy as well as by finitely many local modifications or moves 
listed in Theorem~\ref{th:moves} below (e.g., see \cite[Theorem~4.2]{KS}). 
For a precise formulation of \emph{moves} and their 
\emph{realizability} (or \emph{always-realizability}), the reader is referred to
\cite[\S3]{Sae25} as well.

\begin{theorem}\label{th:moves}
Under a generic homotopy between stable maps of
a closed manifold of dimension $\geq 2$ to a surface, the image of the singular 
point set is modified by isotopy as well as
\begin{itemize}
\item type II fold crossings and their reverse moves,
\item type III fold crossings, 
\item cusp-fold crossings and their reverse moves,
\item birth-death moves,
\item cusp merges and fold merges, and
\item flip-unflip moves.
\end{itemize}
\end{theorem}

To discuss the realizability of the 
moves listed in Theorem~\ref{th:moves}, we need a coorientation
(or normal orientation) convention as follows. 
Suppose that $p$ is a fold singular point of $f$, 
$(t, x_1, x_2, \ldots, x_{m-1})$ are local coordinates around $p$, 
and $(X, Y)$ are local coordinates around $f(p)$ with respect to which $f$ is given by its normal form
as in Definition~\ref{def:fold}.
We choose the coorientation 
of $f(\Sigma(f))$ at $f(p)$ as indicated in Fig.~\ref{fig:coorientation}. 
We emphasize that locally in the standard coordinate chart, the coorientation of the image of the fold curve 
is chosen in the direction of $-Y$. 
 
\begin{figure}[t]
 	\centering
\includegraphics[width=0.4\textwidth]{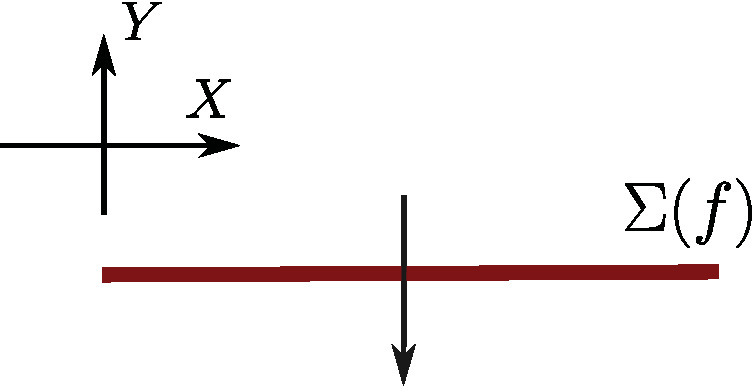}
 	\caption{Coorientation convention at a fold singularity}
 	\label{fig:coorientation}
\end{figure}
 
Now we discuss the six local modifications listed in Theorem~\ref{th:moves}.
 
\subsection*{Type II fold crossing}
The type II fold crossing increases the number of 
self-intersection points of $f(\Sigma(f))$ by two, see Fig.~\ref{fig:II}
from left to right.
Suppose that the two fold curves are of indices $i$ and $j$
as indicated in the figure. 
If $i\le j$, then the fold crossing move is always-realizable by a 
(generic) homotopy. 
The inverse to the fold crossing is realizable if $j<i$. 

\begin{figure}[t]
 	\centering
 	\includegraphics[width=0.6\textwidth]{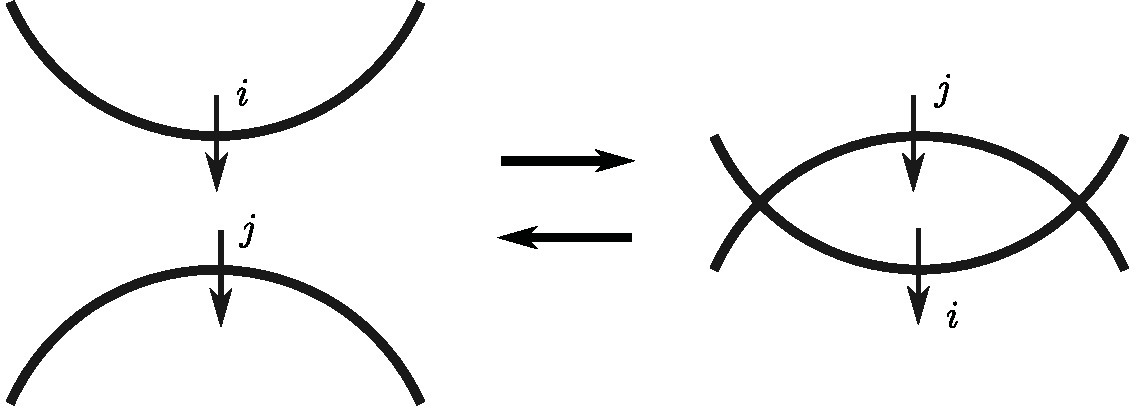}
 	\caption{Type II fold crossing and its reverse move}
 	\label{fig:II}
\end{figure}

\subsection*{Type III fold crossing}
This move occurs at a triple self-intersection point of the images of three arcs of fold singular points. 
Suppose that the three fold curves are of indices $i, j$ and $k$
as in the left of Fig.~\ref{fig:III}. 
The move in the figure from left to right is realizable if 
any of the following three conditions are satisfied (see \cite[Chapter IV, \S2.2, Prop.~2]{Cerf}):
\begin{itemize}
	\item $i+k\ge m$, 
	\item $\max\{i, k\}> j$, or
	\item $i_1=i_2=i_3\ge 2$.
\end{itemize} 
The move from right to left is realizable if 
\begin{itemize}
	\item $i+k\le m-2$,
	\item $\min\{i, k\}< j$, or
	\item $i=j=k\le m-3$. 
\end{itemize}

\begin{figure}[t]
	\centering
	\includegraphics[width=0.6\textwidth]{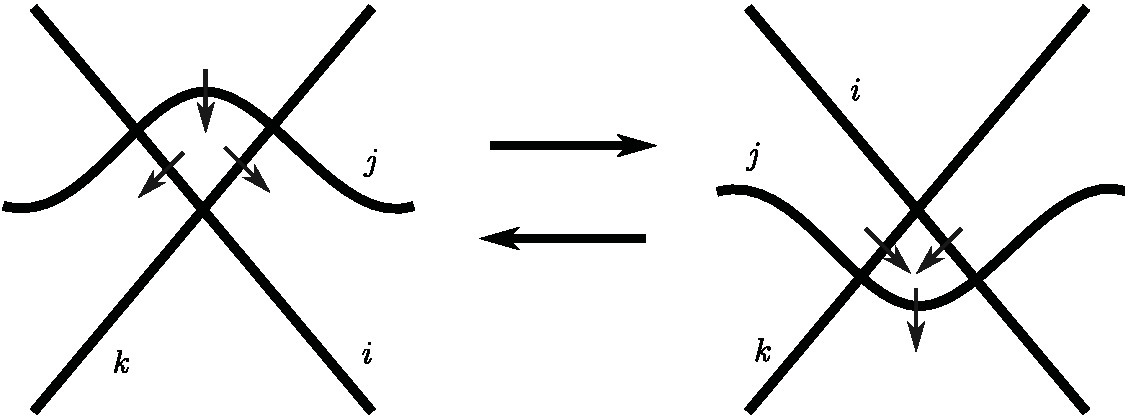}
	\caption{Type III fold crossing}
	\label{fig:III}
\end{figure}

\subsection*{Cusp-fold crossing}
The cusp-fold crossing is shown in  Fig.~\ref{fig:foldcusp},
 from left to right. This
move is realizable if $i > 0$. 
The reverse move is realizable if (see \cite[Chapter IV, \S3.3, Prop. 4]{Cerf})   
\begin{itemize}
	\item $i+j < m-2$, or
	\item $i<j$.
\end{itemize}

\begin{figure}[t]
	\centering
	\includegraphics[width=0.6\textwidth]{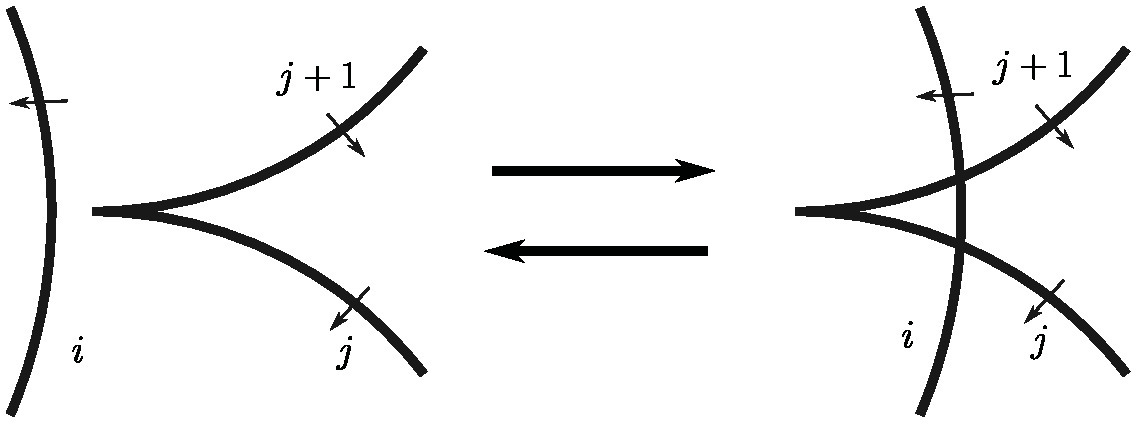}
	\caption{Cusp-fold crossing and its reverse move}
	\label{fig:foldcusp}
\end{figure}

\subsection*{Birth move and death move}
The birth move creates a new component of singular points. 
Such a move is always-realizable,
see Fig.~\ref{fig:wrinkle}.
Its reverse, i.e., the move that 
eliminates a component
%
is always-realizable in the case where one of the fold arcs is of absolute index $0$, 
see \cite{Cerf}. 

\begin{figure}[t]
	\centering
	\includegraphics[width=0.6\textwidth]{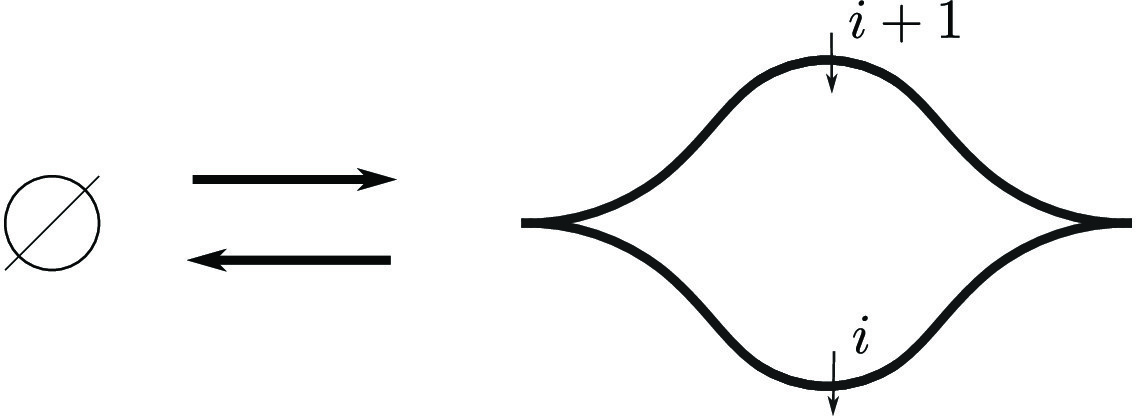}
	\caption{Birth move (from left to right) and death move 
        (from right to left)}
	\label{fig:wrinkle}
\end{figure}

\subsection*{Cusp merge and fold merge}
The cusp merge, depicted in Fig.~\ref{fig:cusp_merge}, from left to right,
is realizable provided that the fiber over the point $\ast$ is path connected.
The fold merge, which is depicted in Fig.~\ref{fig:cusp_merge},
from right to left, is not always-realizable. 

\begin{figure}[t]
	\centering
	\includegraphics[width=0.7\textwidth]{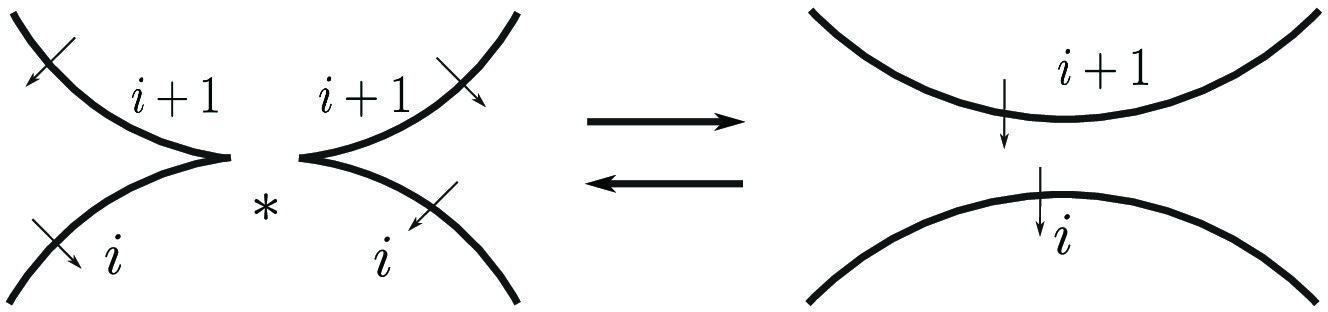}
	\caption{Cusp merge (from left to right) and fold merge (from right to left)}
	\label{fig:cusp_merge}
\end{figure}

\subsection*{Flip and unflip}
The flip move is depicted in Fig.~\ref{fig:swallowtail}, from
left to right. It is
always-realizable for $i<m-1$.
The unflip is depicted in Fig.~\ref{fig:swallowtail}, from
right to left. It
is realizable if $i=0$ or $2\le i\le m-2$, see \cite[Part I, Chapter V, Proposition 1.4]{HW}.  

\begin{figure}[t]
	\centering
	\includegraphics[width=0.5\textwidth]{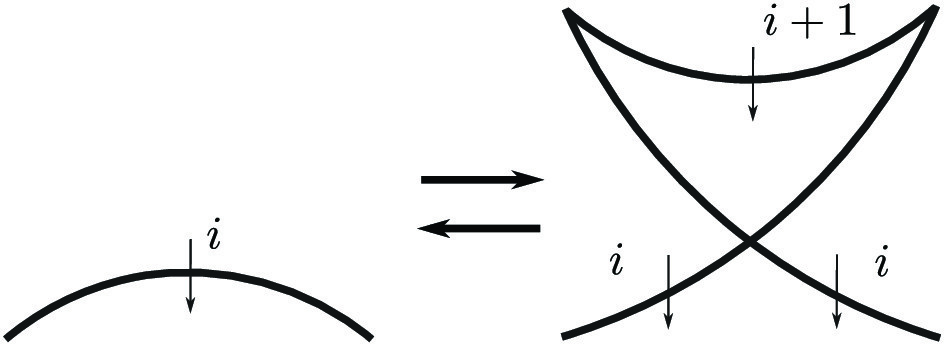}
	\caption{Flip (from left to right) and unflip (from right to left)}
	\label{fig:swallowtail}
\end{figure}

\section{A peculiar open book decomposition of the sphere $S^{2n+1}$}\label{s:obd}

For $n \geq 1$, let $P$ denote the total space of the disc bundle over $S^{n}$ 
associated with the tangent bundle $TS^n$ over $S^n$. 
Specifically, we choose $S^n$ to be the standard unit Riemannian sphere in $\R^{n+1}$, 
and let $P$ be the submanifold of $TS^n$ consisting of points $(x, v)$ with $x \in S^n$ 
and $v \in T_xS^n$ such that the length $||v||$ of the vector $v$ at $x$ is at most $\pi$. 
The projection of this disc bundle is denoted by $\pi_P\co P\to S^n$.

\begin{example}
	In the case $n=1$, the disc bundle $P$ over $S^1$ is diffeomorphic to the cylinder $S^1\times [0, 2\pi]$. We will define a generalized Dehn twist motivated by the standard example. The standard Dehn twist is a diffeomorphism of the cylinder that twists it fixing the boundary pointwise. In coordinates it is given by $(x, t)\mapsto (xe^{it}, t)$. 
\end{example}

We now define a generalized Dehn twist of $P$ relative to the boundary as follows. 
Suppose that $\varphi$ is a self-diffeomorphism of $P$ that is the identity on the boundary.
It takes a fiber $D_x=\pi_P^{-1}(x)$ over a point $x \in S^n$ to a properly embedded disc $\varphi(D_x)$ in $P$. 
Since $\varphi$ preserves $\partial D_x$ point-wise, the projection $\pi_P|_{\varphi(D_x)}$ takes 
the boundary of the disc $\varphi(D_x)$ to the point $x$, and therefore the restriction 
$\pi_P\circ \varphi|_{D_x}$ defines an element in $\pi_{n}(S^n)\cong \Z$. 
We say that the self-diffeomorphism $\varphi$ of $P$ is a \emph{generalized Dehn twist} 
if it is the identity on the boundary, and the map $\pi_P\circ \varphi|_{D_x}$ 
represents a generator of the group $\pi_{n}(S^n)\cong \Z$ for every fiber $D_x$ of $\pi_P$. 
(If we orient $S^n$ and the fibers of $P$ consistently, then we say that
$\varphi$ is a \emph{generalized positive Dehn twist} if
$\pi_P \circ \varphi|_{D_x}$ represents $1 \in \pi_{n}(S^n)\cong \Z$.)

To construct such a diffeomorphism $\varphi$ 
explicitly, let $\gamma \co P \to P$ denote the smooth map that 
takes a tangent vector $v$ at a point $x\in S^n$ to the point $(\ell(1), \ell'(1)) \in P$, where $\ell$ is 
the unique geodesic $\R \to S^n$ such that $\ell(0)=x$, $\ell'(0) = v$, and
$\ell'(t) \in T_{\ell(t)}S^n$ denotes the velocity vector of $\ell$ at $t \in \R$.
We see easily that $\gamma$ is a diffeomorphism, since one can construct its inverse by 
sending $(x, v)$ to $(\ell(-1), \ell'(-1))$.

\begin{lemma}\label{lemma:gamma}
The diffeomorphism $\gamma \co P \to P$ is isotopic to the identity through
diffeomorphisms which may not fix the boundary point-wise.
\end{lemma}

\begin{proof}
For $v \in T_xS^n$ with $||v|| \leq \pi$ and $t \in [0, 1]$, let $\gamma_t(x, v)$ be the point 
$(\ell(t), \ell'(t)) \in P$, where $\ell$ is the unique geodesic as described above.
We see easily that this defines a well-defined smooth geodesic flow $\gamma_t \co P \to P$,
which will be often considered as an isotopy of $P$, in the following argument.
By an argument similar to the above, we see that
this defines a smooth $1$-parameter family of diffeomorphisms.
Furthermore, we see that $\gamma_0$ is the identify of $P$ and that $\gamma_1 = \gamma$.
This completes the proof.
\end{proof}

Now, we define the diffeomorphism $\varphi$ by $\varphi =\tau\circ \gamma$, where $\tau$ is the involution 
that takes a tangent vector $v$ at $x\in S^n\subset \R^{n+1}$ to the tangent vector $-v$ at $-x$. 
Here, we are identifying the tangent space $T_x S^n$ with $\{v \in \R^{n+1}\,:\, x \cdot v = 0\}$, where
``$\cdot$'' denotes the usual Euclidean inner product.
It immediately follows that the diffeomorphism $\varphi$ is the identity on the boundary of $P$, 
and that $\pi_P\circ \varphi|_{D_x}$ represents a generator of the group $\pi_{n}(S^n)\cong \Z$ for every 
fiber $D_x$ of $\pi_P$.

We also construct a Morse function $f \co P \to [-1, 2]$ with exactly two
critical points $p$ and $q$ of indices $0$ and $n$, respectively,
such that $f(p) = -1$, $f(q) = 1$ and $f(\partial P) = 2$. 
Namely, let $\xi \co S^n \to [-1, 1]$ be a standard Morse function on $S^n$ with two critical points
$p$ and $q$ of indices $0$ and $n$, respectively, 
and $\eta$ a Morse--Bott function on $P$ that assigns to each vector $v$ the square $||v||^2$ of its length. 
Then $f = \xi \circ \pi_P+(2 - \xi \circ \pi_P) \cdot \eta/\pi^2$ is a desired Morse function on $P$.
Note that the critical points $p$ and $q \in P$ of $f$ are points in the zero section $S^n$ of $P$. 
The closure of the descending disc of the critical point $q$ coincides with the zero section $S^n$ of $P$, 
while the ascending disc of $q$ is the fiber $D_q=\pi_P^{-1}(q)$, provided that we choose the gradient-like
vector field of $f$ appropriately. 

Let $DP = P_R \cup P_L$ denote the (untwisted) double of the manifold $P$, 
formed from the disjoint union of $P_R$ and $P_L$ by identifying their boundaries by the identity map.
Then, $\varphi^{-1}(D_q)$ in $P_R$ together with $D_q$ in $P_L$ form a sphere $S_1$. 
A copy of $D_q$ in $P_R$ and a copy of $D_q$ in $P_L$ also form a sphere $S_2$.  

\begin{lemma}\label{l:int} 
The sphere $S_1$ is isotopic, in $DP$, to a sphere that intersects $S_2$ transversely at a unique point.  
\end{lemma}

\begin{proof}  
The sphere $S_1$ is isotopic to the sphere $S_1'$ which is the union
of a copy of $\varphi^{-1}(D_p)$ in $P_R$ and a copy of $D_p$ in $P_L$,
since $D_p$ and $D_q$ are isotopic through fibers.
Then, we see that $S_1' \cap S_2$ consists of $\varphi^{-1}(D_p) \cap D_q$ in $P_R$,
since $D_p$ and $D_q$ are disjoint in $P_L$. Since $\varphi = \tau \circ \gamma$
and $\tau^{-1}(D_p) = D_q$, we have only to show that $D_q$ and $\gamma(D_q)$
intersect transversely at $(q, 0)$, the point of the zero
section on the fiber $D_q$. This is obvious by the very definition of $\gamma$.
This completes the proof.
\end{proof}

Let $X$ denote the closed manifold of dimension $2n+1$ with an open book decomposition given by 
the page $P$ and monodromy $\varphi$. 
We will give the description of $X$ as in \cite{Hs}. 
The manifold $X$ can be constructed from $[0, 2\pi] \times P$ by performing a ``vertical'' identification 
$V$ and a ``horizontal'' identification $H$, where
\begin{itemize}
\item $V$ identifies $(2\pi, x)$ with $(0, \varphi(x))$ for every $x\in P$, and
\item $H$ identifies $(t, y)$ with $(s, y)$ for every $y \in \partial P$ and all $t,s\in [0, 2\pi]$. 
\end{itemize}
Let $X_+$ and $X_-$ denote the subspaces of $X$ consisting of the equivalence classes of points in 
$[0, \pi] \times P$ and $[\pi, 2\pi] \times P$, respectively. 
Then $X$ is the union of $X_+$ and $X_-$. 
The interiors of $X_+$ and $X_-$ are disjoint, while the boundaries $\partial X_+$ and $\partial X_-$ are 
diffeomorphic to the double $DP$ of the manifold $P$. 
We write $\partial X_\pm=P^\pm_R \cup P^\pm_L$, where $P^+_R\subset X_+$ stands for $\{0\} \times P$, 
$P^+_L$ and $P^-_L$ stand for $\{\pi\} \times P$, and $P^-_R$ stands for $\{2\pi\} \times P$. 
Thus, the manifold $X$ is obtained from the disjoint union of $X_+$ and $X_-$ by identifying  $P^+_L$ with $P^-_L$ 
by means of the identity map, and by identifying $P^-_R$ with $P^+_R$ by means of $\varphi \co P_R^- \to P_R^+$
(see Fig.~\ref{fig3}).

\begin{figure}[t]
\centering
\includegraphics[width=\linewidth,height=0.25\textheight,
keepaspectratio]{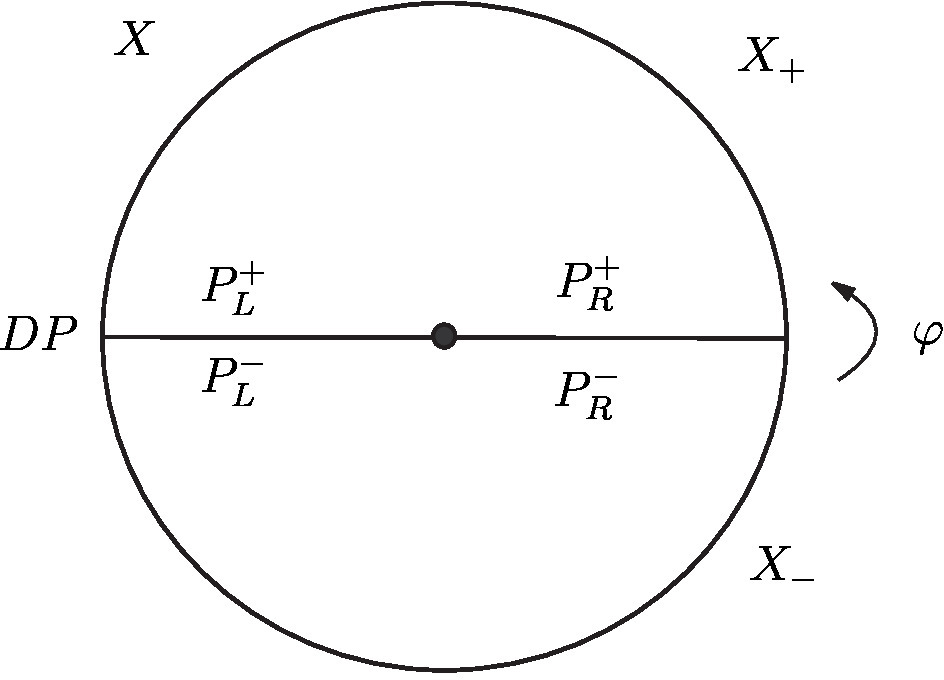}
\caption{Open book decomposition of $X$}
\label{fig3}
\end{figure}

The page $P$ has a handle decomposition that consists of a handle of index $0$, and a handle $h^n$ of index $n$. 
Therefore, the manifold $X$ admits a handle decomposition with a handle of index $0$, a handle of index $n$, 
a handle of index $n+1$ and a handle of index $2n+1$. 
Indeed, the manifold $X_-$ admits a handle decomposition with a handle of index $0$ and 
a handle of index $n$, while $X_+$ admits a handle decomposition with a handle of index $n+1$ 
and a handle of index $2n+1$. Thus, the manifold $X$ admits a handle decomposition as described. 

The attaching sphere of the $(n+1)$-handle in $X$ is the union of the cocore disc $coc(h^n)_L$ of 
the $n$-handle $h^n$ in $P_L^- \cong P$ and the copy of its image $\varphi^{-1}(coc(h^n)_R)$ in $P_R^- \cong P$. 
By Lemma~\ref{l:int}, the attaching sphere $coc(h^n)_L \cup \varphi^{-1}(coc(h^n)_R)$ of 
the $(n+1)$-handle in $X$ is isotopic to a sphere that intersects the belt sphere 
$coc(h^n)_R \cup coc(h^n)_L$ of the $n$-handle in $X$ at a unique point transversely, and therefore 
the $n$-handle and the $(n+1)$-handle in $X$ are canceling. 
We conclude that the manifold $X$ is obtained by gluing two smooth $(2n+1)$-dimensional discs. 
Therefore, $X$ is a smooth manifold $\Sigma^{2n+1}$
homotopy equivalent to a sphere. 
In particular, for $n=1, 2$, $X$ is diffeomorphic to the standard sphere $S^{2n+1}$.

We claim that the monodromy $\varphi$ can be modified, if necessary, in a neighborhood $U$ of a point, 
say $x_0 \in \mathrm{Int}\,P = \mathrm{Int}\,P_R^-$, 
so that the resulting manifold $X$ is diffeomorphic to the standard $(2n+1)$-sphere, 
provided that $n \geq 3$.
Furthermore, we may choose $U$ such that its closure
is contained in $P_R^- \setminus Q$, where $Q$ is the union of the zero section $S^n$ of $P_R^-$
and $\varphi^{-1}(D_q)$.
Indeed, assume that $X$ is an exotic sphere $\Sigma^{2n+1}$. 
It is obtained by gluing two discs $D^{2n+1}_-$ and $D^{2n+1}_+$ by means of a diffeomorphism $\phi\co S^{2n} \to S^{2n}$. 
By modifying $\phi$ by isotopy if necessary, we may assume that $\phi$ is the identity map except on 
a small disc neighborhood $B^{2n}$ of a point.
Let $\phi_B \co B^{2n} \to B^{2n}$ be the restriction of $\phi$ to $B^{2n}$.
Let us embed $B^{2n}$ to a small neighborhood $U$ of a point $x_0$ in the interior of $P_R^-$
such that the closure of $U$ is contained in $P_R^- \setminus Q$.
Then, by modifying $\varphi \co P_R^- \to P_R^+$ on $B^{2n}$ using $\phi_B^{-1}$,
we get a manifold diffeomorphic to $X \sharp (-\Sigma^{2n+1}) \cong S^{2n+1}$, where $-\Sigma^{2n+1}$
is the manifold $\Sigma^{2n+1}$ with the orientation reversed (see \cite{Browder}).

Recall that $P$ admits a Morse function $f$ with a unique critical point $p$ of index $0$ and 
a unique critical point $q$ of index $n$. 
Now, we construct a smooth family of functions on $P$:
\[
f_t\co P\longrightarrow \R
\]
parametrized by $t\in [0, 2\pi]$ such that $f_0=f$, $f_{2\pi}=f \circ \varphi$, $f_t$ is a Morse function 
with a unique critical point of index $0$ with value $-1$, a unique critical point of index $n$ with value $1$, 
and $f_t(\partial P) = 2$ for all $t\in [0, 2\pi]$. 
To this end we will construct an ambient isotopy $g_t$ of $P$ parametrized by $t\in [0, 2\pi]$ 
such that $f_t=f\circ g_t$ is a desired family of Morse functions. 

To construct an ambient isotopy $g_t$, we first consider an isotopy parametrized by $t \in [0, \pi]$ of 
the ascending disc $\varphi^{-1}(D_q) = \gamma^{-1}(\tau^{-1}(D_q))$ of $f\circ \varphi$ to the ascending disc 
$D_q$ of $f$. Such an isotopy can be constructed by using 
the ambient isotopy $\gamma_t$, $t \in [0, 1]$, followed by an isotopy 
taking $\tau^{-1}(D_q)$ to $D_q$, where $\gamma_t$, $t \in [0, 1]$, is the ambient
isotopy of $P$ constructed in the proof of Lemma~\ref{lemma:gamma}.
Furthermore, we may choose an ambient
isotopy $\tau_t$ of $P$ that takes $\tau^{-1}(D_q)$ to $D_q$ so that
it consists of diffeomorphisms of $P$ induced by rotations
of $S^n$ about any axis perpendicular to the line $\overline{pq}$
passing through $p$ and $q$ in $\R^{n+1}$.
Here, note that when $n$ is even, the isotopy thus constructed reverses the
orientations of $\varphi^{-1}(D_q)$ and $D_q$.

Finally, we can modify the isotopies $\gamma_t$ and $\tau_t$ slightly 
so that their concatenation defines a smooth ambient isotopy $g_t$ of $P$ 
parametrized by $t \in [0, \pi]$.

Note that the Morse function $f \circ \varphi \co P \to [-1, 2]$ has critical
points $q$ and $p$ of indices $0$ and $n$, respectively. 
Then, 
we have $f\circ g_{\pi}=f\circ \varphi$ over the ascending disc 
$\varphi^{-1}(D_q)$ of $p$ as well as the closure of the descending disc $S^n$ of $p$. 
Indeed, since $\gamma_t$ and $\tau_t$ preserve
the lengths of the vectors in $P$, we conclude that $f \circ g_{\pi} =
f \circ \varphi$ over $\varphi^{-1}(D_q)$. On the other hand,
$\gamma$ is the identity on $S^n$, while $\tau$ takes every level
set of $f|_{S^n}$ to a level set. Therefore, we have $f \circ g_{\pi}
= f \circ \varphi$ on $S^n$.
Note that we have $g_\pi(Q)=\varphi(Q)$ even though $g_\pi$ may not be isotopic to $\varphi$ on $Q$.

We observe that $P\setminus Q$ is diffeomorphic to $(\partial P \setminus
\varphi^{-1}(\partial D_q)) \times (-1, 2]$ under the diffeomorphism that takes 
a point $x \in P \setminus Q$ to the pair $(y, f\circ g_\pi(x))$, 
where $y$ is the unique point on $\partial P \setminus \varphi^{-1}(\partial D_q)$ that lies 
on the gradient curve of $f\circ g_\pi$ passing through $x$. 
There is an isotopy $g_t$ of $P\setminus Q$ parametrized by $t\in [\pi, 2\pi]$ that takes 
a point $x$ with coordinates $(y, f\circ g_\pi(x))$ to the point with coordinates $(y, f\circ \varphi(x))$. 
Such an isotopy extends over $Q$, since $f\circ g_\pi=f\circ \varphi$ over $Q$, 
and we can also modify the two isotopies $g_t$ for $t\in [0, \pi]$ and $t\in [\pi, 2\pi]$ to 
produce a smooth isotopy of $P$ parametrized by $t\in [0, 2\pi]$.

By replacing $f_t$ by $f_t-2$, we arrange the Morse functions $f_t$ so that $f_t(\partial P)=0$, $t \in [0, 2\pi]$. 
Define the map 
\[
   [0, 2\pi] \times P \longrightarrow \R^2
\]
by $\rho(t, x)=-f_t(x)$ and $\theta(t, x)=t$, where $(\theta, \rho)$ are the polar coordinates on $\R^2$. 
The map thus constructed defines a well-defined smooth map $f_X \co X \to \R^2$.
This is an image simple round fold map with a unique fold circle of index $2n$, and 
a unique fold circle of index $n$ with respect to the inward direction.

Now, let $f \co M \to S$ be an arbitrary image simple fold map into a (possibly
non-orientable) surface $S$.
We first consider the disjoint union $f \cup f_X \co M \cup X \to S$,
where $f_X(X)$ is contained in a small $2$-disc in $f(M) \setminus f(\Sigma(f))$.
We apply the birth move (see 
Fig.~\ref{fig:wrinkle}) to create a wrinkle consisting
of two cusps together with fold arcs of absolute indices $0$ and $1$
in such a way that the image sits just next to $f_X(\Sigma(f_X))$. (For
this and the following, see Fig.~\ref{fig4}. The integers attached to some curves 
indicate the corresponding
indices of the relevant folds with respect to the normal directions
indicated by arrows, see Section~\ref{s:moves}.
Then, we construct an image simple fold map $M \# X \cong M \to S$ 
by the connected sum operation along definite folds, as described in \cite[Lemma~5.4]{Sae93},
in such a way that the image of the singular point set is as depicted in the upper right
figure of Fig.~\ref{fig4}. Note that the fold index $n$ does not change
even if we change the normal direction, since the dimension
of the source manifold is equal to $2n+1$.
We apply the reverse move of type II crossing twice so that the image of the fold circle
of index $n$ gets out of the wrinkle. This is possible
as long as $n < 2n-1$, i.e.\ $n \geq 2$ (see 
Fig.~\ref{fig:II}).
Finally, we can eliminate the wrinkle, which is possible,
since one of the fold arcs has absolute index $0$ (see Fig.~\ref{fig:wrinkle}). 

The resulting fold map $g \co M \to S$ is image simple
and we have $\#|\Sigma(g)| = \#|\Sigma(f)|+1$.

This completes the proof of Theorem~\ref{th:main} for $n \geq 2$.

\begin{figure}[t]
\centering
\includegraphics[width=0.95\linewidth,height=0.5\textheight,
keepaspectratio]{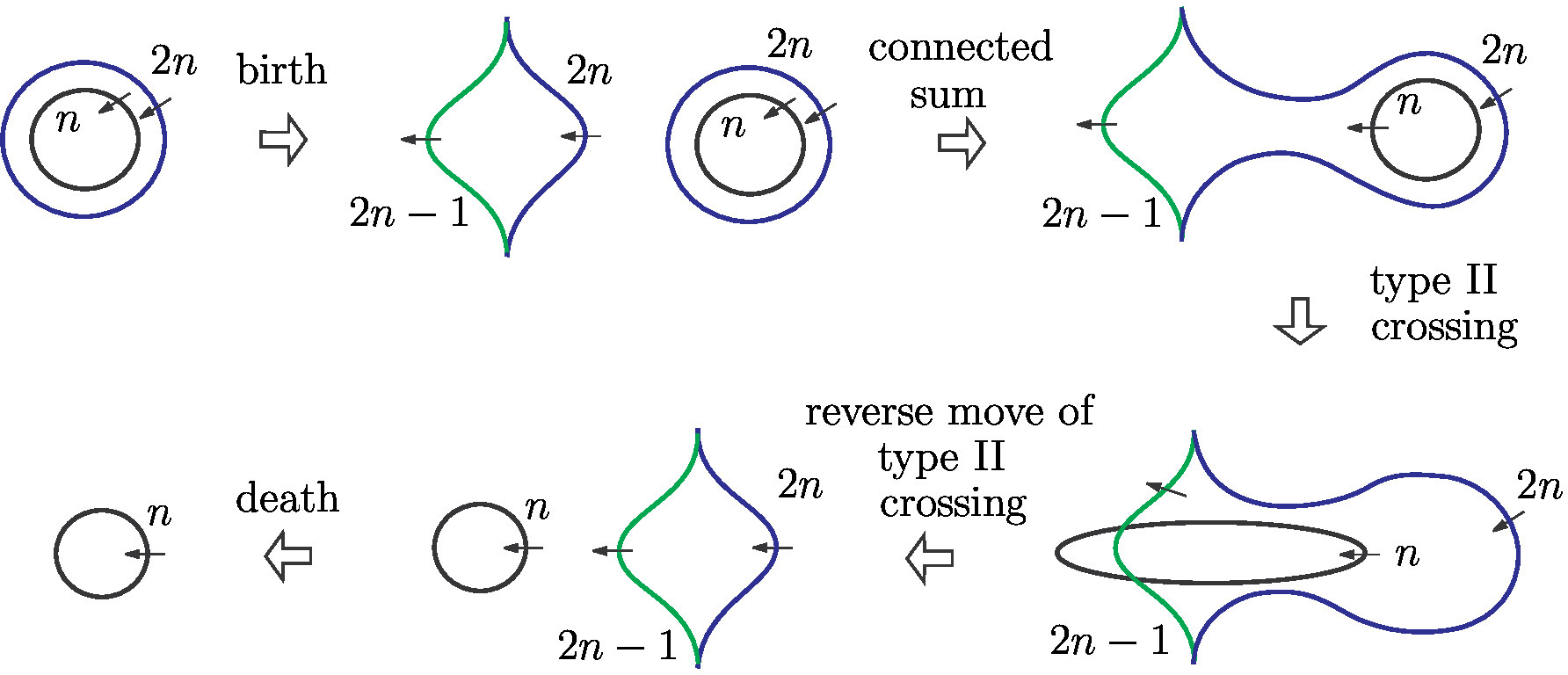}
\caption{Sequence of the singular point set images}
\label{fig4}
\end{figure}

\begin{remark}
(1) We do not know if a similar argument works for $n=1$.
We will give another argument to cover this case in Section~\ref{s:2a}.
On the other hand, if we merge the two cusp points
in the upper right figure of Fig.~\ref{fig4} using
a path whose image is embedded in the left hand side
of the wrinkle, we get an image simple fold map $h \co M \to S$
homotopic to $f$ such that $\#|\Sigma(h)| =
\#|\Sigma(f)| + 3$. This proves Theorem~\ref{th:main}
for the case of $n=1$.

(2) We see easily that $f$ and $g$ are homotopic as continuous maps.
However, we do not know how to construct an explicit homotopy between $f$ and $g$
thus constructed.
\end{remark}

\section{Allowable moves approach}\label{s:2a}

Let us give another proof to Theorem~\ref{th:main} by an approach
similar to that in \cite{Sae19a}. For the moves that we
use in this section, the reader is referred to Section~\ref{s:moves}. 

Suppose $n \geq 2$. Let $f \co M \to S$ be an arbitrary
image simple fold map of a closed $(2n+1)$-dimensional
manifold into a (possibly non-orientable) surface $S$.
We consider a small $2$-disc $\Delta$ contained in $f(M) \setminus f(\Sigma(f))$,
and we modify the map $f$ on a connected component of $f^{-1}(\Delta)$,
following the procedure as described in Fig.~\ref{fig1}.
These figures depict the singular value sets of generic maps inside $\Delta$.
The integers attached to some curves indicate the corresponding
indices of the relevant folds with respect to the normal directions
indicated by arrows.
The lower left figure depicts the singular value set inside the $2$-disc $\Delta$,
i.e.\ the empty set.
Then, we create three wrinkles, using the birth move three times (see Fig.~\ref{fig:wrinkle}). 
Then, by three flip moves, we turn the wrinkles into three bow ties with involved folds of 
indices $n-1$, $n$ and $n+1$ (see Fig.~\ref{fig:swallowtail}).
Note that the index $n-1$ turns into $n+1$ if we reverse the normal direction.
Then, by three cusp merges we merge the three bow ties as in the
upper right figure (see Fig.~\ref{fig:cusp_merge}). 
This is possible, since 
we are modifying the maps only on a connected component
of $f^{-1}(\Delta)$ so that
the relevant fibers over the points on the paths between
the images of the cusps that are merged are connected.
As the three fold curves all have indices $n \geq 2$,
we can perform the type III crossing as in the
lower right figure (see Fig.~\ref{fig:III}).
Finally, we need to perform three unflips. This is possible, 
since the involved folds are of indices $n$ and $n+1$ with
$2 \le n \leq 2n-1$ (see Fig.~\ref{fig:swallowtail}).

\begin{figure}[t]
\centering
\includegraphics[width=\linewidth,height=0.5\textheight,
keepaspectratio]{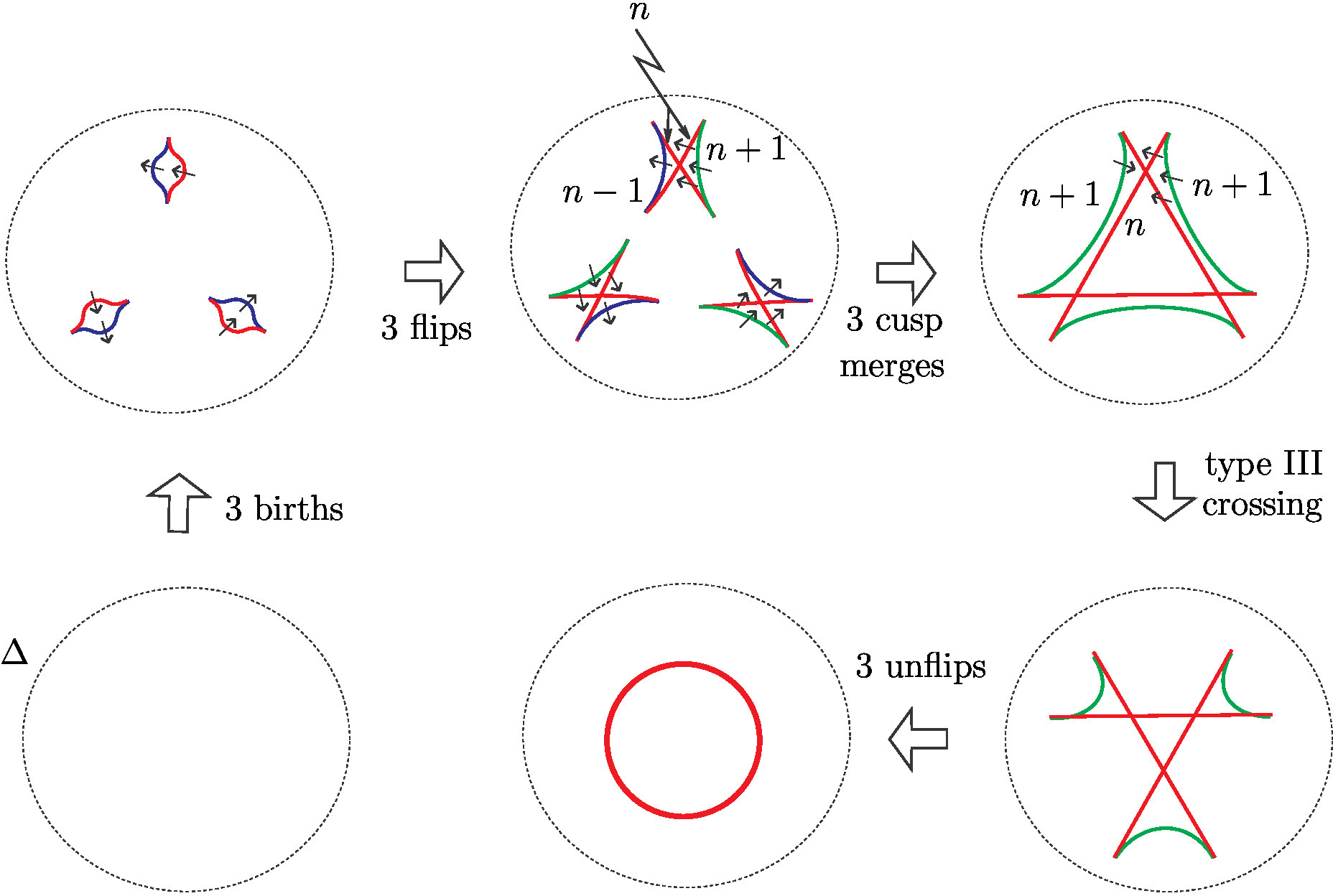}
\caption{Sequence of moves}
\label{fig1}
\end{figure}

The above argument may not work directly for $n=1$. However,
as noted in \cite{Sae19a}, by merging the three pairs of cusps with
appropriate ``twists'', we can guarantee that the
type III crossing can be applied.

This completes the proof of Theorem~\ref{th:main}.

\medskip

This approach allows us to determine the singular point set of the constructed homotopy
as follows.
As we perform three births, we have three $2$-discs as the singular point set
in the first three steps. In the fourth step, we merge the $2$-discs
by three bands in such a way that the resulting compact surface
has exactly one boundary circle whose image
is depicted in the upper right figure of Fig.~\ref{fig1}.
Therefore, the singular point
set for the first four steps must be a M\"{o}bius band.
The rest of the homotopy does not change the topological type
of the singular point set component. Therefore, the resulting
singular point set of the (generic) homotopy is a M\"{o}bius band, 
except for the cylinder part corresponding to the fold
circles of the original image simple fold map $f$.

So, the singular point set of the entire (generic) homotopy $F \co M \times [0, 1] \to
S \times [0, 1]$
is diffeomorphic to the disjoint union of $\#|\Sigma(f)|$
copies of the cylinder together with a M\"{o}bius band.
Note also that $F|_{\Sigma(F)}$ 
has a triple self-intersection point.

This is consistent with the results obtained in \cite{KS}
when the target surface $S$ is orientable 
and $M$ is a closed orientable
manifold of dimension $3$, or $S = \R^2$ or $S^2$
and $M$ is a closed orientable manifold of odd dimension $m$ with
$m > 2$.

Note that in Theorem~\ref{th:main}, the target surface $S$
is possibly non-orientable so that we do not
know if for every homotopy $F$ as in Theorem~\ref{th:main},
$F|_{\Sigma(F)}$ has a triple self-intersection point.

\begin{remark}
Let us give a short remark concerning the first Takase problem formulated in \cite[Problem 3.2]{Sae19}
mentioned in Section~\ref{s:intro}.
The map $f \co S^3 \to \R^2$
constructed in \cite{Sae19a} does not give a direct negative answer to the problem.
However, if we use the procedure as described above, we get a counter
example for every element $n$ of $\pi_3(S^2) \cong \Z$,
since the newly created fold circle is indefinite.
Therefore, the number modulo two of fold components of such image
simple fold maps without definite fold points is not a homotopy invariant in general.
\end{remark}

\section{Even dimensional case}\label{s:even}

The following has been obtained
in \cite[Theorem~1.1]{KS} by using cumulative winding number in
$\frac{1}{2}\Z$. Here, we give another proof,
which seems to be more elementary in nature.

\begin{theorem}\label{thm:even}
Let $M$ be an even-dimensional closed manifold of dimension $m \geq 2$
and $S$ a connected orientable surface. Then, for image simple fold
maps $f \co M \to S$, the parity of the number of connected
components $\#|\Sigma(f)|$ of the singular point set is a homotopy
invariant of $f$.
\end{theorem}

\begin{proof}
Consider an image simple fold
map $f \co M \to S$. Then, $f(\Sigma(f))$ is a finite disjoint
union of embedded closed curves in $S$, and we denote
by $R_1, R_2, \ldots, R_s$ the connected components of
$S \setminus f(\Sigma(f))$.
We also denote by $\chi_i \in \Z$ the Euler characteristic
of the fiber over a point in $R_i$, $i = 1, 2, \ldots, s$.
When $R_i$ is not contained in the image of $f$, we set $\chi_i = 0$.
Note that the regular fibers of $f$ are of even dimension $m-2$
and that the parity of their Euler characteristics
is a homotopy invariant of $f$. We denote it by $\chi_f \in \Z/2\Z$.
Set $r_f =0 \in \Z/4\Z$ if $\chi_f=0$ and set $r_f = 1 \in \Z/4\Z$
otherwise.

Let us consider a point in $S$ crossing a component of $f(\Sigma(f))$.
Then the Euler characteristic of the corresponding fiber always
changes by $\pm 2$. Let $\tilde{R}_0$ (resp.\ $\tilde{R}_1$)
be the union of $R_i$'s such that $\chi_i \equiv r_f \pmod{4}$
(resp.\ $\chi_i \equiv r_f+2 \pmod{4}$).
Note that $S \setminus f(\Sigma(f))$ is the disjoint
union $\tilde{R}_0 \cup \tilde{R}_1$. 

First suppose that $S$ is a closed connected orientable surface.
Then, we have
\begin{eqnarray*}
\chi(M) & = & \sum_{i=1}^s \chi_i \chi(R_i) \\
& \equiv & r_f \chi(\tilde{R}_0) + (r_f+2) \chi(\tilde{R}_1) \pmod{4} \\
& \equiv & r_f \chi(\tilde{R}_0) + (r_f+2) (\chi(S) - \chi(\tilde{R}_0) \pmod{4} \\
& \equiv & 2\chi(\tilde{R}_0) + (r_f+2)\chi(S) \pmod{4}.
\end{eqnarray*}
As the closure of $\tilde{R}_0$ is a compact orientable surface
with boundary $f(\Sigma(f))$, we see that $2\chi(\tilde{R}_0)
\equiv 2\#|\Sigma(f)| \pmod{4}$.
Since $\chi(M)$ and $(r_f + 2)\chi(S) \pmod{4}$ are homotopy invariants of $f$,
we see that the parity of $\#|\Sigma(f)|$ is a homotopy
invariant of $f$.

In the general case where $S$ may not be closed,
for two homotopic image simple fold maps $f$ and $g \co M \to S$,
we consider an appropriate compact sub-surface $\bar{S}$ of $S$
that contains the image of the homotopy between $f$ and $g$.
Let $\widehat{S}$ be the closed orientable surface
obtained by attaching $2$-discs to the boundary circles of $\bar{S}$.
Then, for image simple fold maps $f$ and $g \co M \to \widehat{S}$,
the above argument implies that $\#|\Sigma(f)| \equiv
\#|\Sigma(g)| \pmod{2}$.

This completes the proof.
\end{proof}

As a corollary to the above proof,
we have the following.

\begin{corollary}
Let $M$ be an even-dimensional closed manifold of dimension $m \geq 2$
and $S$ a connected orientable surface. 
Suppose that $M$ admits an image simple fold map 
$f\co M \to S$. Then $\chi(M)$ is even. 
If $S$ is not closed, then we have
$$\#|\Sigma(f)| \equiv \frac{\chi(M)}{2} \pmod{2}.$$
If $S$ is closed, then $\chi(S)$
is also even, and we have
$$\#|\Sigma(f)| \equiv \frac{\chi(M)}{2} + \chi_f\frac{\chi(S)}{2} \pmod{2},$$
where $\chi_f =0 \in \Z/2\Z$ if the regular fibers of $f$
have even Euler characteristics, and $\chi_f = 1 \in \Z/2\Z$
otherwise.
\end{corollary}

\begin{remark}\label{rem:surface}
In the case of an image simple fold map $f$ of a closed orientable surface $M$ to $\R^2$, the components of 
$\Sigma(f)$ can be signed so that the algebraic number of components of $\Sigma(f)$ is 
equal to $\chi(M)/2$. 
Specifically, every component of $f(\Sigma(f))$ bounds a region (actually, a $2$-disc) in $\R^2$. 
Given a component $\gamma$ of $f(\Sigma(f))$, if the number of points
in the preimage of a point in $\R^2$ near $\gamma$ in the bounded region is greater than 
that of a nearby point in the unbounded region, then we say that $\gamma$ is positive; otherwise it is negative.
Suppose $\Sigma(f)$ has $n$ negative components and $p$ positive components. 
We can perform surgeries of index $2$ along the $n$ negative components;
in other words, for each such component, we remove an annular neighborhood 
of the corresponding singular point set component in $M$, and attach two $2$-discs
on which we consider the natural embedding into $\R^2$.
In this way, we can construct an image 
simple fold map of the resulting surface $M'$ with $p$ positive components. Then $M'$ is a disjoint union of 
$p$ copies of the $2$-sphere, and the Euler characteristic of $M' $ is equal to $2p = \chi(M) + 2n$. 
Thus, for the map $f$, the algebraic number of fold components is precisely equal to $\chi(M)/2$. 

Let us now show that even if we allow $g$ not to be image simple, there exists no fold map $g \co S^2 \to \R^2$
with $\# |\Sigma(g)|$ even.
Let $f \co M \to \R^2$
be a fold map of a closed orientable surface $M$ to the plane. Then, $\Sigma(f)$
splits $M$ into compact orientable (possibly disconnected) surfaces $M_+$ and $M_-$ with boundary,
where $f|_{M_+}$ (or $f|_{M_-}$) is orientation preserving (resp.\ reversing).
By \cite{E}, their Euler characteristics coincide. Since
$\chi(M_+) + \chi(M_-) = \chi(M)$ and $\chi(M)$ is even,
we see that $\chi(M_+) = \chi(M_-) = \chi(M)/2$.
Now, we have $\partial M_+ = \partial M_- = \Sigma(f)$ and the number of boundary 
components of $M_+$ (or $M_-$) has the same parity as $\chi(M_+)$  (resp.\ $\chi(M_-)$),
since they are orientable. Therefore, $\# |\Sigma(f)|$ has the same parity as $\chi(M)/2$,
even if $f$ may not be image simple.
\end{remark}

In the case where the target surface is non-orientable,
Theorem~\ref{thm:even} does not hold in general.
Let us construct an explicit example as follows.

\begin{example}
Let $w$ and $z \co S^1 \to \R$ be two Morse functions such that
(refer to Fig.~\ref{fig5})
\begin{itemize}
\item $w$ has exactly two critical points $a_1$ and $a_2$
of indices $0$ and $1$, respectively, with $w(a_1) = -1$ and
$w(a_2) = 1$,
\item $z$ has exactly four critical points $b_1, b_2, b_3$ and $b_4$
of indices $0, 0, 1$ and $1$, respectively, with $z(b_1) = -1$,
$z(b_2) = -1/2$, $z(b_3) = 1/2$ and $z(b_4) = 1$,
\item there exists an orientation reversing involution $\sigma \co S^1 \to S^1$
such that $w \circ \sigma = -w$, $\sigma(a_1) = a_2$
and $\sigma(a_2) = a_1$,
\item there exists an orientation reversing involution $\tau \co S^1 \to S^1$
such that $z \circ \tau = -z$, $\tau(b_1) = b_4$, $\tau(b_2) = b_3$,
$\tau(b_3) = b_2$ and $\tau(b_4) = b_1$.
\end{itemize}

\begin{figure}[t]
\centering
\includegraphics[width=\linewidth,height=0.5\textheight,
keepaspectratio]{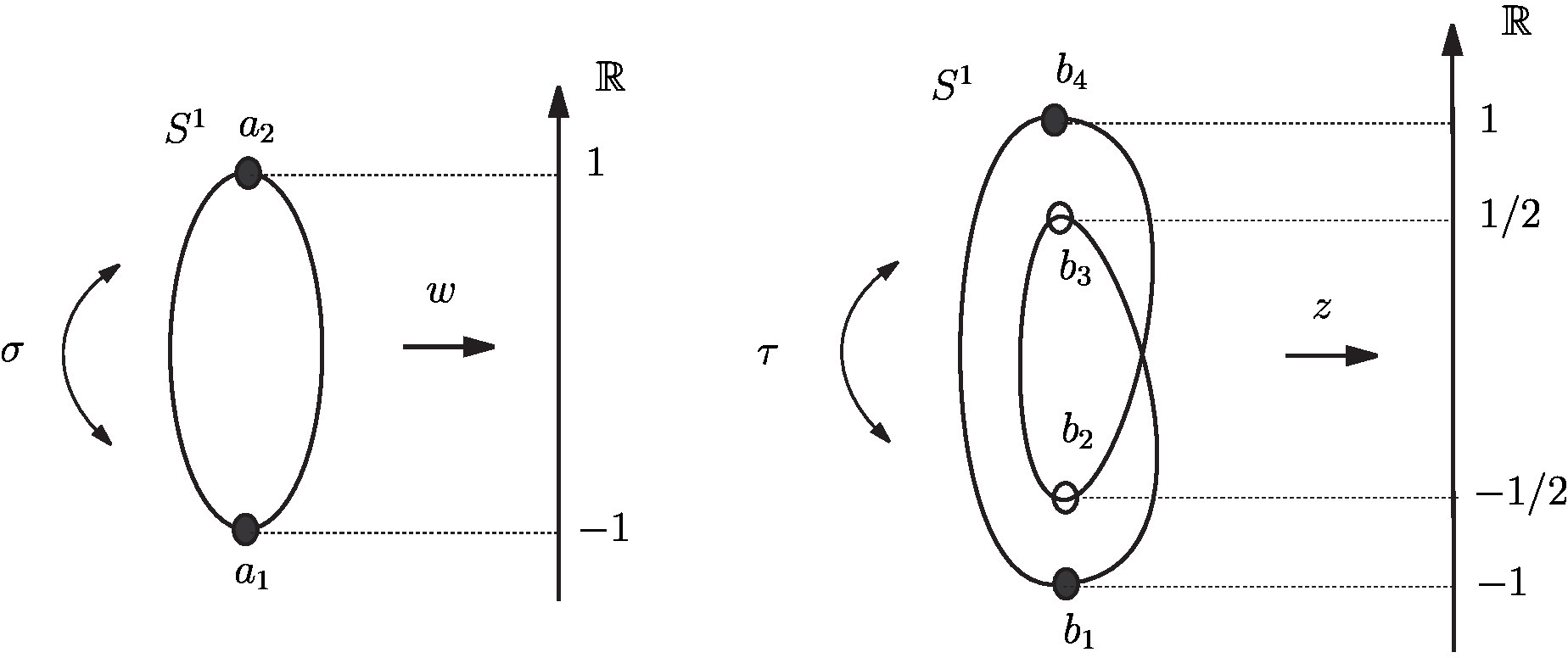}
\caption{Two Morse functions on $S^1$}
\label{fig5}
\end{figure}

Then, $w$ induces a fold map
$$\mathrm{id}_{[0, 1]} \times w \co
[0, 1] \times S^1 \to [0, 1] \times \R,$$
which further induces a fold map
$$\tilde{w} = \mathrm{id}_{S^1} \tilde{\times} w \co S^1 \tilde{\times} S^1
\to S^1 \tilde{\times} \R,$$
where $\mathrm{id}_{[0, 1]}$ is the identity map of the interval
$[0, 1]$, $S^1 \tilde{\times} S^1$ is obtained from $[0, 1] \times S^1$
by identifying $(1, x)$ with $(0, \sigma(x))$ for each $x \in S^1$,
and $S^1 \tilde{\times} \R$ is obtained from $[0, 1] \times \R$
by identifying $(0, t)$ with $(1, -t)$ for each $t \in \R$.
Similarly, we get a fold map
$$\tilde{z} = \mathrm{id}_{S^1} \tilde{\times} z \co S^1 \tilde{\times} S^1
\to S^1 \tilde{\times} \R.$$
We see that $\tilde{w}$ and $\tilde{z}$ are image simple
fold maps from the Klein bottle to the open M\"{o}bius band $\mathcal{M}$
and that $\sharp|\Sigma(\tilde{w})| = 1$, while
$\sharp|\Sigma(\tilde{z})| = 2$. On the other hand,
since $w$ and $z$ are null-homotopic, we see that
$\tilde{w}$ and $\tilde{z}$ are homotopic.
Hence, the parity of the number of singular point set
components is not a homotopy invariant for the image
simple fold maps $\tilde{w}$ and $\tilde{z}$.
\end{example}

The above construction can be generalized to produce functions $w$ and $z\co S^{2n-1}\to \R$,
$n \geq 1$,
that lead to image simple fold maps $\tilde{w} = \id_{S^1} \tilde\times w$ and 
$\tilde{z} = \id_{S^1} \tilde\times z$ of 
the $2n$-dimensional Klein bottle $S^1 \tilde\times S^{2n-1}$ into the open M\"obius band 
$\mathcal{M} = S^1 \tilde\times \R$ such that $\tilde{w}$ has a single fold component, 
while $\tilde{z}$ has two fold components, as follows.

The function $w$ is the standard height function on the unit sphere $S^{2n-1}\subset \R^{2n}$, 
where $n\ge 1$, i.e., $w$ is the restriction of the last coordinate function $x_{2n}$ in $\R^{2n}$ to $S^{2n-1}$. 
Let $\sigma \co S^{2n-1} \to S^{2n-1}$ be the orientation reversing involution defined by
the reflection with respect to the coordinate hyperplane $x_{2n} = 0$. Note that
$w \circ \sigma = -w$.

To define the function $z$, let us observe that $S^{2n-1}$ can be obtained from the disjoint union 
of two solid tori $S^{n-1}\times D^{n}$ and $D^{n}\times S^{n-1}$ by identifying their boundaries 
via the identity map:
\[
  \partial(S^{n-1}\times D^{n})=S^{n-1}\times S^{n-1}
  \stackrel{\id}\longrightarrow S^{n-1}\times S^{n-1}=\partial(D^{n}\times S^{n-1}).
\]  
Here, in order to get $S^{2n-1}$ which is oriented, we need to orient
$S^{n-1}\times D^{n}$ and $D^{n}\times S^{n-1}$ in such a way that the identity
map as above should be orientation reversing.

Since the solid torus $S^{n-1}\times D^{n}$ is a union of a $0$-handle and an $(n-1)$-handle, 
this description yields a handle decomposition of $S^{2n-1}$ that consists of four handles of 
indices $0$, $n-1$, $n$ and $2n$ respectively. In turn, such a handle decomposition gives 
rise to an associated Morse function $z$ on $S^{2n-1}$. There is an involution $\tau$ on 
$S^{2n-1}$ that swaps the two solid tori $S^{n-1}\times D^{n}$ and $D^{n}\times S^{n-1}$. 
By the above remark, $\tau$ should be orientation reversing. Furthermore,
we can choose $z$ so that $z\circ \tau=-z$. 
Then, the rest of the argument is parallel to the above argument.

This proves Theorem~\ref{thm:main2}.

We can construct a generic homotopy between the image
simple fold maps $\tilde{w}$ and $\tilde{z} \co S^1 \tilde{\times} S^{2n-1}
\to \mathcal{M}$ as follows. We first apply the birth move (see
\cite[Fig.~5]{Sae25}) to $\tilde{w}$ in such a way that we have
a wrinkle consisting of 
two cusps together with fold arcs of indices $n-1$ and $n$ near the center circle of $\mathcal{M}$, i.e.,
inside the region $R$ bounded by $\mathcal{M} \setminus \tilde{w}(\Sigma(\tilde{w}))$.
Note that if we reverse the normal direction, the indices turn to $n$ and $n-1$.
Then, we merge the two cusps using a path connecting them in the
source manifold whose image is embedded in $R$.
This is possible for $n > 1$, since the fibers over the points in $R$
are connected. For $n=1$, the fibers are not connected: however,
we can use a similar cusp elimination technique as in \cite{Millett}.
Then, we get an image simple fold map equivalent to $\tilde{z}$.
The singular point set of the (generic) homotopy thus constructed
is the disjoint union of a M\"{o}bius band and a cylinder,
and it has no triple self-intersections.

\section*{Acknowledgment}\label{ack}
The authors would like to express their deep gratitude
and respect, with sadness, to Professor Maria Aparecida Soares Ruas, Cidinha,
for her continual encouragement for pursuing research,
especially in singularity theory and topology.
This paper was respectfully written responding to her kind request
to write something of geometric topology.

The authors would like to thank David Gay
who kindly explained to the authors how to construct generalized
Dehn twists. 
They would like to thank Sergey Melikhov for 
a discussion which led to Remark~\ref{rem:surface}.
They are also thankful to Liam Kahmeyer for discussions of 
the topic.
They would also like to thank the
anonymous referee for valuable comments which improved the
presentation of the paper. The second author would also like to thank 
the Department of Mathematics, Kansas State University,
for the hospitality
during the preparation of the manuscript.

This work has been partially supported by JSPS KAKENHI Grant Numbers 
\linebreak
JP22K18267, JP23H05437.

\end{document}